\documentclass{amsart}

\usepackage{amsthm}
\usepackage{amsfonts}
\usepackage{amsmath,amsthm,amssymb}
\usepackage{comment}
\usepackage{fancybox}
\usepackage{epsf}
\usepackage[dvips]{graphicx}

\newtheorem{theorem}{Theorem}[section]
\newtheorem{lemma}[theorem]{Lemma}
\newtheorem{proposition}[theorem]{Proposition}
\theoremstyle{definition}
\newtheorem{definition}[theorem]{Definition}
\theoremstyle{remark}

\numberwithin{equation}{section}

\begin{document}

\title{Finite type invariants of words and Arnold's invariants}
\author{Mikihiro Fujiwara}
\address{Department of Mathematical and Computing Sciences, 
Tokyo Institute of Technology, Tokyo 152-8552, Japan}
\email{fujiwar5@is.titech.ac.jp}


\begin{abstract}
We define a new finite type invariant for stably homeomorphic class 
of curves on compact oriented surfaces without boundaries and 
extend to a regular homotopy invariant for spherical curves.
\end{abstract}

\maketitle

\section{Introduction}
This paper is motivated by N. Ito's work in  [3]  which is based on
the following three existing theories.

First of all,
V. I. Vassiliev developed in [8] a theory of finite type invariants
of knots, also known as Vassiliev invariants, which is conjectured
to classify knots.
Secondly,
V. Turaev suggested that words and their topology could
be an effective coding for virtual knots in  [6].
And finally,
V. I. Arnold developed a theory of invariants for generic curves
on the plane in  [1, 2]  which obeys some rules under regularly
homotopic change.
We should note that M. Polyak reconstructed Arnold's invariants
in terms of finite type invariants of plane curves.

Ito defined finite type invariants of words called  $SCI_n$  in [3],
and showed that they become a complete invariant for stable
homeomorphism classes and also that Arnold's invariants are
of first order.
The purpose of this paper is to provide a very simple 
frame to define finite type invariants of words by
introducing a new type of intersection, called singular intersection,
which plays intermediate role between an actual and virtual
intersections, and show that our finite type invariants contain all
the information Ito's finite type invariants have, and provide
a clearer relation to Arnold's invariant.

The papers consists of 4 parts.
The first part (Sections 2 - 4) is devoted to the basics.
We define curves, signed words and study the relation between them. 
In the second part (Section 5)
we introduce the new type of intersection, called singular intersection.
In the third part (Sections 6 - 7) 
we define a new type of finite type invariants of words and discuss the structure.
In the fourth part (Sections 8) 
we show how to extend finite type invariants of spherical curves and study relation to Arnold's invariants.

\section{Curves}
A curve will be a generic immersion from an oriented circle to a closed oriented surface. 
Here generic means that the curve has only finite set of self-intersections, 
which are all transversally double points.
Triple points and self-tangencies are not allowed.
Every curve has a regular neighborhood.
This is a narrow immersed band whose core is the curves in the surface.
A pointed curve is a curve endowed with a base point, which is not on the singular points.  

Two curves are stably homeomorphic 
if there is a homeomorphism between their regular neighborhoods 
mapping the first curve onto the second one preserving the orientation of the curves and the surfaces.  
Similarly, two pointed curves are pointed stably homeomorphic 
if two curves are stably homeomorphic preserving their base points.

\section{Signed words}
Let $\mathcal{A}$ be a set of letters.
A word of length $n$ in $\mathcal{A}$ is a mapping from ${\hat n} = \{ 1,2,3,...,n\}$ to $\mathcal{A}$ .
A Gauss word is a word where each letter of $\mathcal{A}$ appears in the word exactly two times.
A Gauss word is a signed word if the letters in the associated word naturally acquire signs $\pm$. 
To simplify notation, we omit signs and put bars on the letters of minuses.  
Thus instead of writing $A^+B^-A^+B^-$ we write $A\bar{B}A\bar{B}$ .
By convention, there is one empty word $\phi$ of length $0$.

Two signed words $w$ in $\mathcal{A}$ , $w^\prime $ in $\mathcal{A}^\prime$ are isomorphic 
if there is a bijection $f:\mathcal{A}\to \mathcal{A}^\prime$ such that $w^\prime = fw$ .
Two signed words are isomorphic if two words are isomorphic preserving sings of letters.
We denote by $\mathcal{W}_{n}$ the set generated by all the isomorphism classes of signed words of length $2n$.  
Put $\mathcal{W}=\cup_{n\geq 0}\mathcal{W}_{n}$.
The isomorphism classes $\mathcal{W}$ of signed word bijectively correspond 
to stable homeomorphism classes of pointed curves on a surface \cite{T2}.

Sub-words of a signed word are given by eliminating some letters from it.
Sub-words of signed words on $\mathcal{A}$ given by eliminating letters $\mathcal{B} \subset \mathcal{A}$ 
are singed words on $\mathcal{A} - \mathcal{B}$.

\begin{definition}
The {\it cyclic equivalence} $\sim$ of signed words is defined by the relation by generated by
\[A^\pm xA^\pm y \sim xA^\mp yA^\mp\]
where $A$ is a letter, $xy$ is a sub-word.
\end{definition}
Let us denote cyclic equivalence class of signed word by $W_{n}/\sim$, 
which bijectively corresponds to stable homeomorphism class of curves on surface \cite{T2}.

\section{Curves as signed words}
Consider a pointed curve on a surface.  
Label its crossings in an arbitrary way by different letters $A_1,A_2,...,A_n$ 
where $n$ is the number of the crossings.
The Gauss word of the curve is obtained by moving along the curve starting at the base point 
and writing down the letters as we encounter them, 
finishing when we get back to the base point.
The resulting word, $w$ on the alphabet $\mathcal{A}=\{ A_1,A_2,...,A_n\}$ contains every letter twice.
To make $w$ a signed word, consider the crossing of the curve labeled $A_i$.
If, when moving along the curve as above, we first traverse this crossing from the bottom-left to the top-right,
then $A_i^-$, otherwise $A_i^+$, see Figure \ref{fig:cross}, 
where the orientation of the ambient surface is counterclockwise.
The trivial curve corresponds to the signed word $\phi$.
We assign to our curve the class of this signed word in $\mathcal{W}_{n}$.

A different choice of the labeling of the crossing gives an isomorphic signed word.  
If the curve is changed by a stable homeomorphism, then the associated signed word does not change, 
since it is defined by the behavior of the curve in its regular neightborhood.

\begin{figure}[h]
\begin{center}
\includegraphics{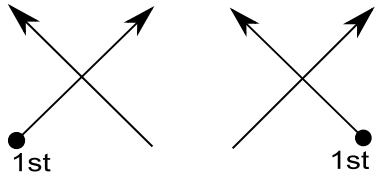}
\caption{On the left $A_i^-$ and on the right $A_i^+$}
\label{fig:cross}
\end{center}
\end{figure}

\section{Singular curves}
\subsection{Virtual crossings}
We may allow {\it virtual} (or {\it unlabeled}) crossings, 
and denote it by a crossing with a small circle around it.
Such crossings do not contribute at all to the associated signed word.
The idea is not that a virtual crossing is just an ordinary graphical vertex,
but rather that the virtual crossing is not actual crossings.

If we draw a non-planar curve in the plane,
it necessarily acquires virtual crossings. 
These crossings are not a part of the structure of the curve itself.
They are artifacts of the drawn of the graph in the plane.

Let $c$ be a curve on surface, and $w$ be a signed word corresponding to $c$.
Getting sub-words of $w$ is equal to changing double points of $c$ to virtual crossings.

\subsection{Singular crossings}
We introduce the crossings called {\it singular} crossings 
as intermediate crossings between actual crossings and virtual crossings, denote as a crossing with a small box around it.  
A singular curve is a curve with singular crossings.

We call a signed word corresponding to a singular curve a singular signed word.
To get signed words corresponding to singular curves, 
label a singular crossing by the letter $A_i^*$ instead of $A_i$.
The letters such as $A^*$(or $\bar{A^*}$) are called singular letters.
For a signed word $w$ without singular letter, 
denote $w^*$ the word changing all letters of $w$ to singular letters saving signs.

\begin{figure}[h]
\begin{center}
\includegraphics{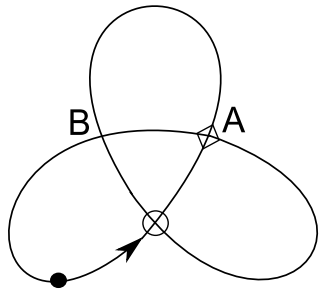}
\caption{The singular curve corresponding to the signed word $\bar{A}^*B\bar{A}^*B$}
\label{fig:singular}
\end{center}
\end{figure}

\section{Definition of finite type invariants}

Let $\mathcal{X}_n$ denote the set of isomorphism class of singular signed words with $n$ singular letters.
By definition, $\mathcal{X}_0=\mathcal{W}$ is the set of non-singular signed words.
Put $\mathcal{X}=\cup_{n\geq 0}\mathcal{X}_n$.

\begin{definition}
A cyclic equivalence invariant $v$ of signed words is a function of a signed word 
in a field $\mathbb{F}$, e.g. $\mathbb{Q}$, $\mathbb{R}$ or $\mathbb{C}$,
which takes equal values on cyclically equivalent signed words.  
In other words, it is a mapping $v$ : $\mathcal{W}/\sim \to \mathbb{F}$
from $\mathcal{W}$ to $\mathbb{F}$.
\end{definition}

\begin{definition}
Given a cyclic equivalence invariant $v$, 
we define its prolongation $\tilde{v}$ by the rule
\begin{eqnarray*}
\tilde{v}\mid_{\mathcal{X}_0} &=& v\\
\tilde{v}(A^*xA^*y) &=& \tilde{v}(AxAy)-\tilde{v}(xy). 
\end{eqnarray*}
\end{definition}

\begin{definition}
A function $v$ is a finite type invariant of order $n$ 
if its prolongation vanishes on all singular words with more than $n$ singular letters:
\[ v(\mathcal{X}^{n+1})=0\]
(and hence $v(\mathcal{X}^m)=0$ for all $m > n$).
\end{definition}

{\bf Notation:} $\mathcal{V}_n$ stands for the set of $\mathbb{F}$-valued invariants of order $\geq n$, 
it will be vector space over $\mathbb{F}$.

\section{Structure of finite type invariant of words}

\begin{proposition}
$\mathcal{V}_0 = \{ {\rm consts} \}$ and hence {\rm dim} $\mathcal{V}_0=1$.
\end{proposition}
\begin{proof}
By definition, $v \in \mathcal{V}_0$ means that 
the value of $v$ on any signed word with at least one singular letter is $0$.
Due to skein relation
\[v(AxAy)-v(xy) = v(A^*xA^*y) = 0 \]
which means that the value of $v$ does not change when any letters adds (or removes).
Therefore, $v$ is a constant function.
\end{proof}

\begin{proposition}
{\rm dim} $\mathcal{V}_1=2$.
\end{proposition}
\begin{proof}
By a similar argument as above, among singular signed word $w \in \mathcal{X}_1$ 
the value is invariant under adding (or removing) letters, 
and depends only on the value of $(AA)^*$ in this case.
Ont the other hand, once we choose the value of $(AA)^*$, 
we get a finite type invariant of degree $1$.
It follows from an easy argument
that a finite type invariant $v$ of degree $1$ on any signed word $w$ 
can be presented by linear sum,
\[v(w) = a_0v(\phi) + a_1(w)v((AA)^*) ,\]
where $a_0$ is a constant and $a_1(w)$ is the number of letters,
which is a finite type invariant of degree $1$.
\end{proof}

\begin{lemma}
Let $w_n$ be a signed words of length $2n$ ($n=0,1,2,...$), 
and $v$ is a finite type invariant of degree $1$.  Then,
\[v(w_n) = v(\phi) + nv((AA)^*).\]
\end{lemma}
\begin{proof}
The statement is true for $n=0$, since $v(w_0) = v(\phi)$.
Assume the statement is true for $n=m$, i.e.,
\[v(w_m) = v(\phi ) + mv((AA)^*).\]
Then by the equation a letter of $w_{m+1}$, $v(A^*xA^*y) = v(AxAy) - v(xy)$,we have
\[v(w_{m+1}) = v((AA)^*) + v(w_m) = v((AA)^*) + mv((AA)^*) = (m+1)v((AA)^*).\]
Thus we have
\[v(w_n) = v(\phi) + nv((AA)^*).\]
\end{proof}

\begin{proposition}
{\rm dim} $\mathcal{V}_2=6$.
\end{proposition}
\begin{proof}
By a similar argument as above, among singular signed word $w \in \mathcal{X}_2$ 
the value is invariant under adding (or removing) letters;
there is four possible values, 
$v((AABB)^*)$,$v((AA\bar{B}\bar{B})^*)$,$v((A\bar{B}\bar{B}A)^*)$,$v((ABAB)^*)$ in this case,
since there is no relation.
The value of a finite type invariant of degree $2$ on any signed word $w$ 
can be presented by linear sum,
\begin{eqnarray*}
v(w) = a_0v(\phi) + a_1(w)v((AA)^*) + b_1(w)v((AABB)^*) + b_2(w)v((AA\bar{B}\bar{B})^*)\\
 + b_3(w)v((A\bar{B}\bar{B}A)^*) + b_4(w)v((ABAB)^*).
\end{eqnarray*}
\end{proof}

\begin{theorem}\label{th:comp}
A finite type invariant is a complete invariant for cyclic equivalence class of signed word.
\end{theorem}
\begin{proof}
Let $w$ be a signed word of length $2n$, and $v$ be a finite type invariant of degree $n$.
The finite type invariant $v(w)$ has the information of all sub-words of $w$.
\end{proof}

\begin{theorem}
A finite type invariant is a complete invariant 
for stable homeomorphism classes of curves on compact oriented surfaces, without boundaries.
\end{theorem}
\begin{proof}
The facts that each cyclic equivalence class of signed words is corresponding to a stably homeomorphism class of curves
and Theorem \ref{th:comp} imply this theorem.
\end{proof}

Also, following lemma derives from the definition of invariant $SCI_n$ in [3],
where $SCI_n$ is a invariant for immersed curves.

\begin{lemma}
A finite type invariant of degree $n$ is equivalent to $\sum_{k=0}^{n} SCI_k$. 
\end{lemma}
\begin{proof}

Let $\Gamma$ be a immersed curve with $n$ crossings, 
and $w$ be a signed word corresponding to $\Gamma$.
$SCI_m(\Gamma)$ has the information of all sub-words of $w$ of length $2m$.
$\sum_{k=0}^{n}SCI_n(\Gamma)$ has the information of all sub-words of $w$, so does a finite type invariant $v(w)$.
They are complete invariants.
\end{proof}

\section{Finite type inveriants for spherical curves}

Let $\mathcal{W}_s$ be a set $\{ w \in \mathcal{W} \mid w ${\rm \ represents a curve on the sphere}$\}$ and 
$\mathcal{W}_s^*$ be $\mathcal{W}_s/\sim \cup (\mathcal{W}-\mathcal{W}_s)$.

\begin{definition}
A signed word invariant for $\mathcal{W}_s^*$ is a function of a signed word 
which takes equal values on cyclically equivalent signed words corresponding to spherical curves.
In other words, it is a mapping $v:\mathcal{W}_s^* \to \mathbb{F}$ from $\mathcal{W}_s^*$ to a field $\mathbb{F}$.
\end{definition}

Let $v$ be a finite type invariant of degree 2.
For a signed word $w$ for curves on the sphere, 
$v$ can be presented as a linear combination as follows,
\begin{eqnarray*}
v(w) &=&a_0v(\phi) + a_1(w)v((AA)^*) + b_1(w)v((AABB)^*) \\
& &+ b_2(w)v((AA\bar{B}\bar{B})^*)+ b_3(w)v((A\bar{B}\bar{B}A)^*) + c_1(w)v((ABAB)^*)\\
& &+ c_2(w)v((A\bar{B}A\bar{B})^*) + c_3(w)v((\bar{A}B\bar{A}B)^*) + c_4(w)v((\bar{A}\bar{B}\bar{A}\bar{B})^*).
\end{eqnarray*}

\begin{definition}[Regular homotopy of signed words] 
Three regular homotopy moves on signed words are defined as follows.
\begin{enumerate}
\item $xy \to A^\pm B^\mp xA^\pm B^\mp y$ .
$x$,$y$ are words without letters $A$,$B$.
Note that if $xy$ is a signed word on $\mathcal{A}$, 
then $A^\pm B^\mp xA^\pm B^\mp y$ is a signed word on $\mathcal{A}\cup \{A,B\}$.
\item $xy \to A^\pm B^\mp xB^\mp A^\pm y$.
$x$,$y$ are words without letters $A$,$B$.
\item $A^\pm B^\pm xA^\pm C^\pm yB^\pm C^\pm z \to B^\pm A^\pm xC^\pm A^\pm yC^\pm B^\pm z$.
$x$,$y$,$z$ are words without letters $A$,$B$,$C$.
\end{enumerate}
\end{definition}

These moves correspond to moves of curves illustrated in Figure \ref{fig:move}.
\begin{figure}[h]
\begin{center}
\includegraphics{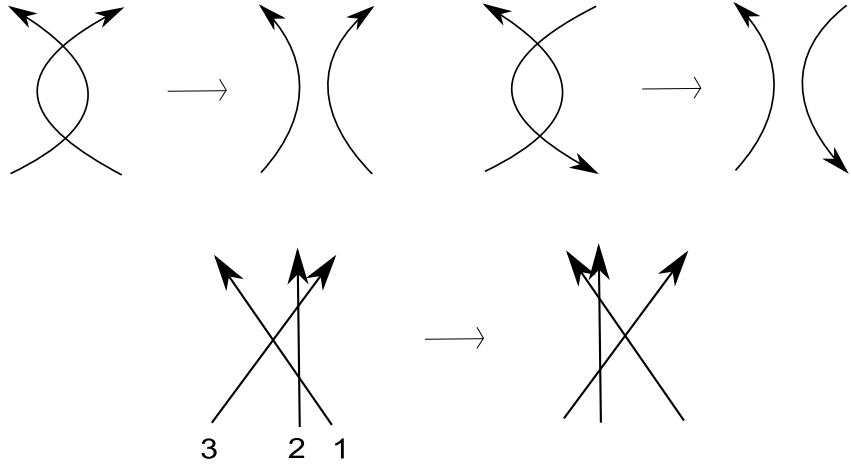}
\caption{Regular homotopy moves}
\label{fig:move}
\end{center}
\end{figure}

The invariants $J_s^+$,$J_s^-$,$St_s$ of generic spherical curves were introduced in \cite{P}.
These invariants are additive with respect to connected sum of curves and 
independent of the choice of orientation for the curve.
$J_s^\pm$ and $St_s$ are defined by their behavior the moves and inverse moves in Figure \ref{fig:move}.

\begin{definition}[Arnold's basic invariants]
$J_s^+$ does not change under the second and third regular homotopy moves or inverse moves
but increases (resp. decrease) by $2$ under the first regular homotopy move (resp. inverse move).

$J_s^-$ does not change under the first and third regular homotopy moves or inverse moves
but decrease (resp. increase) by $2$ under the second regular homotopy move (resp. inverse move).

$St_s$ does not change under the first and second regular homotopy moves or inverse moves
but increase (resp. decrease) $1$ under the third regular homotopy move (resp. inverse move).

Normalizing conditions for Arnoldfs invariants are given as follow.
Let $w_i$  be a signed word $A_1A_1A_2A_2...A_iA_i$
 (or $\bar{A_1}\bar{A_1}\bar{A_2}\bar{A_2}...\bar{A_i}\bar{A_i}$)  
corresponding to a plane curve $K_i$ , $i=1,2,... $, 
\begin{eqnarray*}
J_s^+(K_i) &=& \frac{(i-1)^2}{2}, \\
J_s^-(K_i) &=& \frac{(i-2)^2}{2} - \frac{3}{2},\\
St_s(K_i) &=& -\frac{(i-1)^2}{4}.
\end{eqnarray*}
\end{definition}

\begin{theorem}
Arnold's basic invariants are finite type invariants of degree $2$.
\end{theorem}

\begin{proof}
Let $v_s$ be a finite type invariant of signed words for spherical curves of degree $2$,
and 

\[v_s(\phi)=0,\ v_s((AA)^*)=p,\]
\[v_s((AABB)^*)=v_s((A\bar{B}\bar{B}A)^*)=-v_s((AA\bar{B}\bar{B})^*)=q,\]
\[v_s((ABAB)^*)=v_s((\bar{A}\bar{B}\bar{A}\bar{B})^*)=-v_s((A\bar{B}A\bar{B})^*)=-v_s((\bar{A}B\bar{A}B)^*)=r,\]

Let $\Gamma$ be a curve corresponding to a signed word $w$.  
By using Polyak's formulation of the Arnold's basic invariants \cite{P},
we represent the Arnold's basic invariants as
\begin{eqnarray*}
J_s^+(\Gamma ) &=& v_s(w) \hspace{1em} (p=-\frac{1}{2},\ q=1,\ r=-3),\\
J_s^-(\Gamma ) &=& v_s(w) \hspace{1em} (p=-\frac{3}{2},\ q=1,\ r=-3),\\
St_s(\Gamma ) &=& v_s(w) \hspace{1em} (p=\frac{1}{4},\ q=-\frac{1}{2},\ r=\frac{1}{2}).
\end{eqnarray*}
That is to say, the finite type invariants $v_s$ parameterized by three specific vectors $(p, q, r) = $
$(-\frac{1}{2},1,-3)$,$(-\frac{3}{2},1,-3)$,$(\frac{1}{4},-\frac{1}{2},\frac{1}{2})$ are
equivalent to the Arnold's basic invariants $J_s^+$, $J_s^-$ and $St_s$.
\end{proof}

The invariant can be extended to a invariant $J^\pm$,$St$ for plane curves if rotation number $i$ of curves is defined,
by $J^\pm = J_s^\pm - \frac{1}{2}i^2$ , $St = St_s + \frac{1}{4}i^2$.

\bibliographystyle{plain}

\begin{thebibliography}{10}
\smallskip
\bibitem{Ar1} V. I. Arnold, \textit{Plane curves, their invariants, perestroikas and classifications}, 
Adv. Soviet Math. \textbf{21}, Amer. Math. Soc., Providence, RI (1994), 33--91.
\smallskip
\bibitem{Ar2} V. I. Arnold, \textit{Topological invariants of plane curves and caustics}, 
Univ. Lecture Ser. \textbf{5}, Amer. Math. Soc., Providence, RI (1994).
\smallskip
\bibitem{I} N. Ito, \textit{On a finite type invariant giving complete
classification of curves on surface},
Math.GT/0803.2082.
\smallskip
\bibitem{P} M. Polyak, \textit{Invariants of curves and fronts via Gauss diagrams},
Topology \textbf{37} (1998), 989--1009.
\smallskip
\bibitem{T1} V. Turaev, \textit{Curves on surface, charts, and words}, 
Geom. Dedicata \textbf{116} (2005), 203--236.
\smallskip
\bibitem{T2} V. Turaev, \textit{Knots and words}, Int. Math. Res. Not. (2006), 
Art. ID 84098, 1--23.
\smallskip
\bibitem{T3} V. Turaev, \textit{Lecture on topology of words}, 
Jpn. J. Math. \textbf{2} (2007) 1--39
\smallskip
\bibitem{V} V. A. Vassilliev, \textit{Cohomology of knot spaces, Theory of singularities and its applications}, 
Adv. Soviet Math. \textbf{1}, Amer. Math. Soc. , Providence, RI (1990), 23--69.
\end{thebibliography}

\end{document}